\documentclass{article}

\usepackage{amsmath,amsthm,amssymb,graphicx,tikz,fancyhdr,hyperref}

\newtheorem{dfn}{Definition}[section] 
\newtheorem{thm}[dfn]{Theorem}

\newtheorem{rem}[dfn]{Remark}
\newtheorem{lemma}[dfn]{Lemma}

\begin{document}

%%%%%%%%%%%%%%%%% Author's  name  for running head %%%%%%%%%%%%%%%%%%%%
\author{Kasra Alishahi\\
Sharif University of Technology\\
Erfan Salavati\\
Institute for Research in Fundamental Sciences (IPM)}

\title{A Sufficient Condition for Absolute Continuity of Infinitely Divisible Distributions}

\date{}

\maketitle

\begin{abstract}
We consider infinitely divisible distributions with symmetric L\'evy measure and study the absolute continuity of them with respect to the Lebesgue measure. We prove that if $\eta(r)=\int_{|x|\le r} x^2 \nu(dx)$ where $\nu$ is the L\'evy measure, then $\int_0^1 \frac{r}{\eta(r)}dr <\infty$ is a sufficient condition for absolute continuity. As far as we know, our result is not implied by existing results about absolute continuity of infinitely divisible distributions.
\end{abstract}

\section{Introduction}

Infinitely divisible distributions constitute an interesting class of random variables which are important from both the theoretical and practical point of view. The problem of absolute continuity of infinitely divisible distributions is well studied but no necessary and sufficient condition on the L\'evy measure has been found yet. See the references \cite{Sato}, \cite{Sato-2}, \cite{Kallenberg}, \cite{Orey} and \cite{Nourdin}.

In this article we consider the real valued infinitely divisible distributions with a symmetric L\'evy measure $\nu$.

If the infinitely divisible distribution has a non-zero Gaussian part, its distribution is obviously absolute continuous. Also the drift has no effect on the absolute continuity.

Hence we consider an infinitely divisible distribution without drift and Gaussian part and with L\'evy measure $\nu$. We call such a distribution an infinitely divisible distribution with characteristics $(0,0,\nu)$.

Our purpose is to provide a sufficient condition in terms of $\nu$ under which the distribution is absolutely continuous with respect to the Lebesgue measure. We denote the Lebesgue measure on $\mathbb{R}$ by $\lambda$.

The main result of this article is Theorem~\ref{main_theorem}.

\begin{rem}
	As far as we know, the result of Theorem~\ref{main_theorem} is not implied by existing results about absolute continuity of infinitely divisible distributions. The most well known such result, as stated in \cite{Sato}, Proposition 28.3, provides the sufficient condition $\lim\inf_{r\to 0}\frac{\eta(r)}{r^{2-\alpha}}>0$ for some $\alpha\in (0,2)$. One can verify that $\eta(r)=r^2 (1+\log\frac{1}{r})^{1+\epsilon}$ for any $\epsilon>0$, satisfies our condition but does not satisfy the above condition.
\end{rem}

\section{Results}

We will use the following lemma:

\begin{lemma}\label{lemma}
Let $\lambda$ be the Lebesgue measure on $\mathbb{R}$ and let $\mu$ be a Borel measure on $\mathbb{R}$. For $a\in \mathbb{R}$ let $\mu_a$ be the push-forward of $\mu$ with translation by $a$. Then $\mu \ll  \lambda$ if and only if
\[ \lim_{a\to 0} \| \mu_a - \mu \|_{TV} = 0\]
where $\|.\|_{TV}$ is the total variation norm.
\end{lemma}

\begin{proof}
To prove the if part, assume $\lambda(A)=0$. Hence for any $x\in\mathbb{R}$, $\lambda(A+x)=0$. By Fubiny's theorem,
\begin{multline}\label{equation:proof of lemma_Fubini}
0 = \int_\mathbb{R} \lambda(A+x) d\mu(x) = \int_\mathbb{R} \int_\mathbb{R} 1_A(x+y) d\lambda(y) d\mu(x) \\
=  \int_\mathbb{R} \int_\mathbb{R} 1_A(x+y) d\mu(x) d\lambda(y) = \int_\mathbb{R} \mu_y(A) d\lambda(y)
\end{multline}

On the other hand, by hypothesis, the function $y\mapsto \mu_y(A)$ is continuous at $y=0$ and since it is nonnegative, hence it follows from \eqref{equation:proof of lemma_Fubini} that $\mu(A)=\mu_0(A) = 0$.

For the only if part, note that if $\mu \ll  \lambda$, then by the Radon-Nikodym theorem, $d\mu(x) = f(x) d\lambda(x)$ for some $f\in L^1(\mathbb{R})$ and hence $d\mu_a(x) = f(x+a) d\lambda(x)$ which implies that 
\[ \| \mu_a - \mu \|_{TV} = \| f(a+.)-f(.)\|_{L^1}\| \]
and the last term tends to 0 as $a\to 0$ since $f$ belongs to $L^1(\mathbb{R})$.

\end{proof}

We state our main theorem in terms of the auxiliary function

\[ \eta(r) = \int_{0\le x \le r} x^2 \nu(dx)\]

\begin{thm}\label{main_theorem}
If $\int_0^1 \frac{r}{\eta(r)} dr <\infty$ then the distribution of $X_1$ is absolutely continuous with respect to the Lebesgue measure.
\end{thm}

We will state and prove some lemmas in order to prove Theorem~\ref{main_theorem}.

%\begin{lemma}
%It suffices to prove the Theorem~\ref{main_theorem} in the case that $\nu$ has  support in $[-1,1]$.
%\end{lemma}

%\begin{proof}
%For general $\nu$, let $\nu = \nu_0+\nu_1$ where $\nu_0$ and $\nu_1$ are symmetric and have supports in $[-1,1]$ and $(-\infty,-1]\cup [1,\infty)$. If $\mu$, $\mu_0$ and $\mu_1$ are the infinitely divisible distributions corresponding to $\nu$, $\nu_0$ and $\nu_1$, then we have $\nu= \n_0 * \nu_1$. On the other hand, $\nu_0$ itself satisfies the assumption of Theorem~\ref{main_theorem} and has support in $[-1,1]$, hence is absolutely continuous, which implies that $\nu$ is also absolutely continuous.
%\end{proof}

%From now on, we assume that support of $\nu$ is in $[-1,1]$.

The main idea of the proof of Theorem~\ref{main_theorem} is a coupling of L\'evy processes. In fact let $X_t$ and $Y_t$ be L\'evy processes with characteristic $(0,0,\nu)$ starting respectively at 0 and $a$. Without loss of generality we assume $a>0$. If we denote the distribution of $X_1$ by $\mu$, then the distribution of $Y_1$ is $\mu_a$.

By a coupling between $X_t$ and $Y_t$, we mean a two dimensional process $(\bar{X}_t,\bar{Y}_t)$ where $\bar{X}_t$ has the same distribution as $X_t$ and $\bar{Y}_t$ has the same distribution as $Y_t$. In the following, we provide a coupling between $X_t$ and $Y_t$ and use it to satisfy the conditions of lemma~\ref{lemma}.

We denote the left limit of a cadlag process $X_t$ at $t$ by $X_{t-}$ and let $\Delta X_t = X_t - X_{t-}$.

Let $\bar{X}_t=X_t$ and
\[ \bar{Y}_t = a+\sum_{s\le t} \sigma(s) \Delta X_s \]
where
\[ \sigma(s)=\left\{ {\begin{array}{*{20}{c}}
  1 & \frac{|X_{s-}-Y_{s-}|}{2}<|\Delta X_s| \\ 
   \xi(s) & otherwise 
\end{array}} \right.\]
and $\xi(s)$ are independent random variables with values $\pm 1$ and probability $\frac{1}{2}$ and are independent of $X_t$.

In other words, whenever $\bar{X}_t$ has a jump greater than half of the distance of $\bar{X}_t$ and $\bar{Y}_t$, $\bar{Y}_t$ does the same jump and otherwise $\bar{Y}_t$ does a random jump with the same length but in random direction. It is obvious that the distribution of $\bar{Y}_t$ is the same as $Y_t$ and also by method of construction of $\bar{Y}_t$, we have always $\bar{X}_t\le\bar{Y}_t$.

Let
\[ Z_t = \bar{Y}_t - \bar{X}_t \]
Note that by the way we have constructed the coupling, $Z_t$ is always non-negative and the jumps of $Z_t$ satisfy $| \Delta Z_s| \le Z_{s^-}$.

Let
\[ \tau_a = \inf \{ t: Z_t = 0 \} \]
\[ \bar{\tau}_a = \inf \{ t: Z_t \notin (0,1) \} \]
Note that by the way we have defined $Z_t$, we have $Z_{\bar{\tau}_a} \le 2$.

\begin{lemma}\label{lemma:limit}
\[ \lim_{t\to\infty} Z_t =0,\quad a.s. \]
\end{lemma}
\begin{proof}
$Z_t$ is a non-negative martingale, hence it has a limit $Z_\infty$, as $t\to\infty$. On the other hand, by the assumption made on $\eta$, for any $\epsilon>0$ we have $\nu((0,\epsilon))>0$. Hence jumps of size greater than $\epsilon$ occur in $X_t$ with a positive rate. This implies that if $Z_\infty=\alpha \ne 0$, then $Z_t$ has infinitely many jumps greater than any $\epsilon>\frac{\alpha}{2}$ which contradicts the its convergence. Hence $Z_\infty = 0\quad a.s$.
\end{proof}

We define an auxiliary function $g:[0,\infty)\to\mathbb{R}$ by letting
\[ g(x) = \int_x^1 \int_y^1 \frac{1}{\eta(r)} dr dy \]

\begin{lemma}\label{lemma:g}
If $\int_0^1 \frac{r}{\eta(r)} dr <\infty$ then $g$ is defined on $[0,\infty)$. Moreover, it is differentiable on $(0,\infty)$ and its derivative is absolutely continuous and $g^{\prime\prime} = \frac{1}{\eta} \quad a.e $. Moreover, for every $x,y\in[0,\infty)$, we have
\[ g(y)-g(x) \ge g^\prime(x) (y-x) + \frac{1}{2} \frac{1}{\eta(x)} (y-x)^2 1_{y<x} \]
\end{lemma}

\begin{proof}
By Fubiny's theorem,
\[ g(x) = \int_x^1 \int_0^r \frac{1}{\eta(r)} dy dr = \int_x^1 \frac{r}{\eta(r)} dr \]
hence $g$ is finite and differentiable everywhere and its derivative is absolutely continuous and $g^{\prime\prime} = \frac{1}{\eta} \quad a.e $.

To prove the last claim, note that since $g^\prime=-\int_x^1 \frac{1}{\eta}$ is increasing hence $g$ is convex and therefore,
\[ g(y)-g(x) \ge g^\prime(x) (y-x) \]
If $y<x$, by the integral form of the Taylor's remainder theorem we have
\[ g(y)-g(x) = g^\prime(x) (y-x) + \int_y^x (t-y) g^{\prime\prime}(t) dt \]
where since $g^{\prime\prime}(t)$ is decreasing we have $g^{\prime\prime}(t) \ge \frac{1}{\eta(x)}$ which implies that in the case that $y<x$,
\[ g(y)-g(x) \ge g^\prime(x) (y-x) + \frac{1}{2} \frac{1}{\eta(x)} (y-x)^2\]
and the proof is complete.
\end{proof}

\begin{lemma}\label{lemma:tau_bar}
If $\int_0^1 \frac{r}{\eta(r)} dr <\infty$ then
\[ \lim_{a\to 0} \mathbb{E} \bar{\tau}_a = 0\]
\end{lemma}
\begin{proof}
We have by Lemma~\ref{lemma:g},
\begin{multline} \label{equation:Ito's formula}
	g(Z_{t \wedge \bar{\tau}_a}) = g(a) + \sum_{s\le t\wedge \bar{\tau}_a} \Delta g(Z_s)\\
	\ge g(a) + \sum_{s\le t\wedge \bar{\tau}_a} g^\prime(Z_{s-})\Delta Z_s + \frac{1}{2} \sum_{s\le t\wedge \bar{\tau}_a} \frac{1}{\eta(Z_{s-})}  (\Delta Z_s)^2 1_{\Delta Z_s<0}
\end{multline}

Since the second term on the right hand side of~\ref{equation:Ito's formula} is a martingale, we have
\begin{equation} \label{equation:Ito's formula2}
\mathbb{E} g(Z_{t \wedge \bar{\tau}_a}) \ge g(a) +  \frac{1}{2} \mathbb{E} \sum_{s\le t\wedge \bar{\tau}_a} \frac{1}{\eta(Z_{s-})}  (\Delta Z_s)^2 1_{\Delta Z_s<0}.
\end{equation}

Noting that since the jumps of $Z_s$ are independent and have symmetric distribution, we have
\[  \mathbb{E} \sum_{s\le t\wedge \bar{\tau}_a} \frac{1}{\eta(Z_{s-})}  (\Delta Z_s)^2 1_{\Delta Z_s<0} = \frac{1}{2} \mathbb{E} \int_0^{t\wedge \bar{\tau}_a} \frac{1}{\eta(Z_{s-})} \eta(Z_s) ds =  \frac{1}{2} \mathbb{E} (t \wedge \bar{\tau}_a)\]

Substituting in~\eqref{equation:Ito's formula2} implies

\[	\frac{1}{4} \mathbb{E} (t \wedge \bar{\tau}_a) \le \mathbb{E}g(Z_{t \wedge \bar{\tau}_a}) - g(a) \]

Now letting $t\to \infty$ and noting that $Z_s$ is uniformly bounded by 2 for $t\le \bar{\tau}_a$, we find that
\[	\frac{1}{4} \mathbb{E} (\bar{\tau}_a) \le \mathbb{E}g(Z_{\bar{\tau}_a}) - g(a) \]
Now let $a\to 0$. By continuity of $g$, we have $g(a)\to g(0)$. On the other hand, \[ \mathbb{P}(Z_{\bar{\tau}_a}\ge 1) \le \mathbb{E}(Z_{\bar{\tau}_a}) = a\to 0 \]
where we have used optional sampling theorem. Now we can write,
\begin{multline}
\mathbb{E}g(Z_{\bar{\tau}_a}) = \mathbb{E}g(Z_{\bar{\tau}_a}; Z_{\bar{\tau}_a}=0) + \mathbb{E}g(Z_{\bar{\tau}_a};Z_{\bar{\tau}_a}\ge 1) \\
\le g(0) + 2 \mathbb{P}(Z_{\bar{\tau}_a}\ge 1) \to g(0)
\end{multline}
Hence we find that,
\[ \lim_{a\to 0} \mathbb{E} \bar{\tau}_a = 0\]
\end{proof}

\begin{lemma}
If $\int_0^1 \frac{r}{\eta(r)} dr <\infty$ then $\lim_{a\to 0} \mathbb{P} (\tau_a\ge 1) = 0$.
\end{lemma}
\begin{proof}
We have
\[ \mathbb{P} (\tau_a\ge 1) \le \mathbb{P} (\bar{\tau}_a\ge 1) + \mathbb{P} (\bar{\tau}_a \le 1, Z_{\bar{\tau}_a}\ge 1) \]
The first term on the right hand side is less than or equal $\mathbb{E} \bar{\tau}_a$ and the second term is less than or equal
\[ \mathbb{P} (\sup_{0\le t\le 1} Z_t \ge 1)\]
which by martingale maximal inequalities is less than or equal $\mathbb{E} Z_1 = a$. Hence,
\[ \mathbb{P} (\tau_a\ge 1) \le \mathbb{E} \bar{\tau}_a + a \]
and by Lemma~\ref{lemma:tau_bar} the statement is proved.
\end{proof}

We are now ready to prove Theorem~\ref{main_theorem}.

\begin{proof}[Proof of Theorem~\ref{main_theorem}]
If $(\bar{X}_t,\bar{Y}_t)$ is the coupling introduced above,

\begin{eqnarray*}
\|\mu_{X_1}-\mu_{Y_1}\|_{TV} &=&  \|\mu_{\bar{X}_1}-\mu_{\bar{Y}_1}\|_{TV}\\
& \le & \mathbb{P}(\bar{X}_1\ne \bar{Y}_1)\\
&=&\mathbb{P}(\tau_a >1) \to 0
\end{eqnarray*}
Hence the conditions of lemma~\ref{lemma} are satisfied and we conclude that $\mu_{X_1}$ is absolutely continuous with respect to Lebesgue measure.

\end{proof}

\end{document}